\documentclass[10pt]{article}
\usepackage{amsfonts}
\usepackage{amsmath}
\usepackage{mathrsfs}
\usepackage{mathrsfs, amscd,amssymb,amsthm,amsmath,bm,graphicx,psfrag,subfigure}

\makeatletter

\renewcommand{\@seccntformat}[1]{{\csname the#1\endcsname}{\normalsize .}\hspace{.5em}}
\makeatother

\def \[{\begin{equation}}
\def \]{\end{equation}}
\newtheorem{thm}{Theorem}[section]

\newtheorem{claim}{Claim}
\newtheorem{remark}{Remark}
\newtheorem{lem}[thm]{Lemma}
\newtheorem{cor}[thm]{Corollary}

\newenvironment{kst}
{\setlength{\leftmargini}{2\parindent}
 \begin{itemize}
 \setlength{\itemsep}{-1.1mm}}
{\end{itemize}}

\newenvironment{wst}
{\setlength{\leftmargini}{1.5\parindent}
 \begin{itemize}
 \setlength{\itemsep}{-1.1mm}}
{\end{itemize}}

\begin{document}
\setlength{\baselineskip}{16pt}
\begin{center}{\Large \bf Further analysis on the total number of subtrees of trees\footnote{Financially supported by
the National Natural Science Foundation of China (Grant No. 11071096) and the Special Fund for Basic Scientific Research of Central Colleges (CCNU11A02015).}}

\vspace{2mm}

{\large Shuchao Li\footnote{E-mail: lscmath@mail.ccnu.edu.cn (S.C.
Li), wang06021@126.com (S.J. Wang)},\ Shujing Wang}\vspace{2mm}

{\small Faculty of Mathematics and Statistics,  Central China Normal
University, Wuhan 430079, P.R. China}\vspace{1mm}
\end{center}
\vspace{2mm}

\noindent {\bf Abstract}: We study that over some types of trees with a given number of vertices, which trees
minimize or maximize the total number of subtrees. Trees minimizing (resp. maximizing) the total number of subtrees
usually maximize (resp. minimize) the Wiener index, and vice versa. Here are some of our results:\ (1)\, Let $\mathscr{T}_n^k$
be the set of all $n$-vertex trees with $k$ leaves, we determine the maximum (resp. minimum) value of the total number of subtrees of trees
among $\mathscr{T}_n^k$ and characterize the extremal graphs.\ 
(2)\, Let $\mathscr{P}_n^{p,q}$ be the set of all $n$-vertex trees, each of which has a $(p,q)$-bipartition, we determine
the maximum (resp. minimum) value of the total number of subtrees of trees among $\mathscr{P}_n^{p,q}$ and characterize the extremal graphs.\ (3)\, Let $\mathscr{A}_n^q$ be the set of all $q$-ary trees with $n$ non-leaf vertices, we determine the minimum value of the total number of subtrees of trees
among $\mathscr{A}_n^q$ and identify the extremal graph.

\vspace{2mm} \noindent{\it Keywords}: Subtrees;
Leaves; Diameter; Bipartition; Wiener index; $q$-ary tree

\vspace{2mm}


\vspace{4mm}

 {\setcounter{section}{0}
\section{\normalsize Introduction}

We consider only simple connected graphs (i.e. finite, undirected graphs
without loops or multiple edges).  Let $G=(V_G, E_G)$ be a graph
with $u,v\in V_G$, $d_G(u)$ (or $d(u)$ for short) denotes the
degree of $u$; the \textit{distance} $d_G(u,v)$ is defined as the length of the shortest path between
$u$ and $v$ in $G$; $D_G(v)$ denotes the sum
of all distances from $v$. The \textit{eccentricity} $\varepsilon(v)$ of a
vertex $v$ is the maximum distance from $v$ to any other vertex.

Throughout the text we denote by $P_n,\, K_{1,n-1}$ the path and
star on $n$ vertices, respectively. $G-v,\, G-uv$ denote
the graph obtained from $G$ by deleting vertex $v \in V_G$, or edge
$uv \in E_G$, respectively (this notation is naturally extended if
more than one vertex or edge is deleted). Similarly,
$G+uv$ is obtained from $G$ by adding vertex
edge $uv \not\in E_G$. For $v\in V_G,$ let
$N_G(v)$ (or $N(v)$ for short) denote the set of all the adjacent vertices of $v$ in $G.$
The \textit{diameter} diam$(G)$ of a graph is the maximum
eccentricity of any vertex in the graph.
We refer to vertices of degree 1 of a tree $T$ as \textit{leaves}
(or \textit{pendant vertices}), and the edges incident to leaves are called \textit{pendant edges}. The unique path connecting two vertices $v, u$ in $T$
will be denoted by $P_T(v, u)$. For a tree $T$ and two vertices $v, u$ of $T$, the \textit{distance} $d_T(v, u)$ between them counts the number
of edges on the path $P_T(v, u)$.

Let
$$
W(T)=\frac{1}{2}\sum_{v\in V_T}D_T(v)
$$
denote the \textit{Wiener index} of $T,$ which is the sum of distances of all unordered pairs of
vertices. This topological index was introduced by Wiener \cite{14}, which has been one of the most widely used descriptors in quantitative structure-activity relationships. Since the majority of the chemical applications of the Wiener index deal with chemical compounds with acyclic molecular
graphs, the Wiener index of trees has been extensively studied over the past years; see \cite{15,3G,4,5,24} and the references there for
details.


Given a tree $T$, a \textit{subtree} of $T$ is just a connected induced subgraph of $T$. The number of subtrees as well as related subjects
has been studied. Let $T$ denote a tree with $n$ nodes each of whose non-pendant vertices has degree at least three, Andrew and Wang \cite{17} showed that the average number of nodes in the subtrees of $T$ is at least $\frac{n}{2}$ and strictly less than $\frac{3n}{4}$. Sz\'ekely and Wang \cite{20} characterized the binary trees with $n$ leaves that have the greatest number of subtrees. Kirk and Wang \cite{19} identified the tree, for a given size and such that the vertex degree is bounded, having the greatest number of subtrees.  Sz\'ekely and Wang \cite{21} gave a formula for the maximal number of subtrees a binary tree can possess over a given number of vertices.  They also show that caterpillar trees (trees containing a path such that each vertex not belonging to the path is adjacent to a vertex on the path) have the smallest number of subtrees among binary trees. Yan and Ye \cite{25} characterized the tree with the diameter at least $d$, which has the maximum number of subtrees, and they characterized the tree with the maximum degree at least $\Delta$, which has
the minimum number of subtrees. For some related results on the enumeration of subtrees of trees, one may also see Sz\'{e}kely and Wang \cite{02,03} and Wang \cite{12}. Consider the collection of rooted labeled trees with $n$ vertices, Song \cite{S-C-W} derived a closed formula for the number of these trees in which the child of the root with the smallest label has a total of $p$ descendants. He also derived a recurrence relation for the number of these trees with the property that for each non-terminal vertex $v$, the child of $v$ with the smallest label has no descendants.

It is well known that the Wiener index is maximized by the path and minimized by the star among general trees with
the same number of vertices. It is also known that the counterparts of these simple results for
the number of subtrees do exist. In fact, 
it is interesting that the Wiener index and the total number of subtrees of a tree share exactly the same extremal structure (i.e. the tree that maximizes/minimizes the corresponding index) among trees with a given number of vertices and maximum degree, although the values of the indices are in no general functional correspondence. On the other hand, an acyclic molecule can be expressed by a tree in
quantum chemistry (see \cite{3G}). Obviously, the number of subtrees of a tree can be regarded as a topological index.
Hence, Yan and Ye \cite{25} pointed out that to explore the role of the total number of subtrees in quantum chemistry is an interesting topic.  Motivated by the work of \cite{19,20,21,17,25}, in this paper we continue to study some types of trees which minimize or maximize the total number of subtrees.

Let $\mathscr{T}_n^k$ be the set of all $n$-vertex trees with $k$
leaves  ($2\leq k\leq n-1$). A \textit{spider} is a tree with at most one vertex of degree more than 2, called the \textit{center} of the spider (if no vertex of degree
more than two, then any vertex can be the center). A \textit{leg} of a spider is a path from the center to a vertex of degree 1.
Let $T_n^k$ be an $n$-vertex tree with $k$ legs satisfying all the lengths of $k$ legs, say $l_1, l_2, \ldots, l_k$, are almost equal lengths, i.e.,
$|l_i-l_j|\le 1$ for $1 \le i, j \le k.$ 
It is easy to see that $T_n^k\in \mathscr{T}_n^k$ and $l_i+l_j \in \{2¡¤\lfloor \frac{n-1}{k}
\rfloor , \lfloor \frac{n-1}{k} \rfloor + \lceil \frac{n-1}{k} \rceil ,
2¡¤\lceil \frac{n-1}{k} \rceil \}$, where $1 \le i, j \le k.$ Let $P_k(a,b)$ be a tree obtained by attaching $a$ and $b$ pendant vertices to the two pendant vertices of $P_k$, respectively. In particular, if $k=1$, then $P_k(a,b)=K_{1,a+b}$. It is straightforward to check that $P_{n-k}(\lfloor\frac{k}{2}\rfloor,\lceil\frac{k}{2}\rceil) \in \mathscr{T}_n^k$. It is known that the Wiener index among $n$-vertex trees with $k$ pendant vertices is minimized by $T_n^k$ and is maximized by $P_{n-k}(\lfloor\frac{k}{2}\rfloor,\lceil\frac{k}{2}\rceil)$; see Dobrynin, Entringer, Gutman \cite{15}.
We are going to show the counterparts of these results for the number of subtrees, which also gives a confirm answer for a conjecture proposed by Sz\'{e}kely and Wang in \cite{21}.
\begin{thm}
Among ${\mathscr T}_n^k$ with $n\ge 2$.
\begin{wst}
\item[{\rm (i)}]Precisely the graph $T_n^k$, which has
$(\lfloor\frac{n-1}{k}\rfloor+1)^i(\lceil\frac{n-1}{k}\rceil+1)^j+i{\lfloor\frac{n-1}{k}\rfloor+1 \choose{2}}+j{\lceil\frac{n-1}{k}\rceil+1\choose{2}}$ subtrees,
 maximizes the total number of subtrees among $\mathscr{T}_n^k$, where $i+j=k$ and $n-1\equiv j\pmod{k}$.
\item[{\rm (ii)}]Precisely the graph $P_{n-k}(\lfloor\frac{k}{2}\rfloor,\lceil\frac{k}{2}\rceil)$, which has
$(2^{\lfloor\frac{k}{2}\rfloor}+2^{\lceil\frac{k}{2}\rceil})(n-k-1)+2^k+k+{n-k-1\choose{2}}$ subtrees,
 minimizes the total number of subtrees among $\mathscr{T}_n^k$.
\end{wst}
\end{thm}

Let $\mathscr{T}_{n,d}$ denote the set of all $n$-vertex trees of
diameter $d$. Let $\hat{T}_{n,k}^d$ be the $n$-vertex tree obtained
from $P_{d+1}=v_1v_2\ldots v_dv_{d+1}$ by attaching $n-d-1$ pendant
edges to $v_k$; see Fig. 1.
\begin{figure}[h!]
\begin{center}
\psfrag{a}{$u_1$}\psfrag{b}{$u_2$}\psfrag{1}{$v_1$}\psfrag{2}{$v_2$}
\psfrag{3}{$v_{k-1}$}\psfrag{4}{$v_k$}\psfrag{5}{$v_{k+1}$}\psfrag{6}{$v_d$}\psfrag{7}{$v_{d+1}$}
  \psfrag{c}{$u_{n-d-1}$}\psfrag{d}{$d$}\psfrag{e}{$k-1$}
  \psfrag{f}{$k+1$}\psfrag{g}{$d+1$}
  \includegraphics[width=60mm]{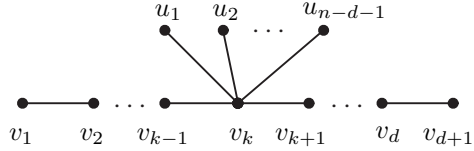}\\
  \caption{Tree $\hat{T}_{n,k}^d.$ }
\end{center}
\end{figure}
\begin{thm}
For any $n\ge 2$, precisely the graph $\hat{T}_{n,i}^d$, which has
$$2^{n-d-1}\left(\left\lfloor \frac{d}{2}\right \rfloor+1\right)\left(\left\lceil \frac{d}{2} \right\rceil+1\right)+ {\lfloor \frac{d}{2} \rfloor+1\choose{2}}{\lceil \frac{d}{2} \rceil+1\choose{2}}+n-d-1$$ subtrees,
 maximizes the total number of subtrees among $\mathscr{T}_{n,d}$, where $i=\lfloor\frac{d}{2}\rfloor+1$ or $i=\lceil\frac{d}{2}\rceil+1$.
\end{thm}

Let $G$ be a connected bipartite graph with $n$ vertices. Hence its vertex set can be partitioned into
two subsets $V_1$ and $V_2$, such that each edge joins a vertex in $V_1$ with a vertex in $V_2$. Suppose that $V_1$ has
$p$ vertices and $V_2$ has $q$ vertices, where $p+q = n$. Then we say that $G$ has a $(p, q)$-\textit{bipartition} $(p \le q)$.
Denote by $\mathscr{P}_n^{p,q}$ the class of trees with $n$ vertices, each of which has a $(p, q)$-bipartition $(p + q = n)$.
Consider a star $K_{1,p}$ with $p + 1$ vertices and attach $q-1$ pendant edges to a non-central vertex of the
star $K_{1,p}$. The resulting tree with $p + q$ vertices has a $(p, q)$-bipartition. Denote the resulting tree by
$D(p, q)$; see Fig. 2. Obviously, $D(p, q)\in \mathscr{P}_n^{p,q}$. We call $D(p, q)$ a \textit{double star}. If $q \geq p \geq 3$, suppose
that $B(p, q)$ is the tree obtained from $D(p-1, q)$ by attaching a pendant edge to one of the vertices of
degree one which join the vertex of degree $q$ in $D(p- 1, q)$ (see Fig. 2). If $q \ge p = 2$, we assume that
$B(2, q)$ is the tree obtained from the path $P_4$ by attaching $q- 2$ pendant edges to an end vertex of $P_4$
(see Fig. 2).
\begin{figure}[h!]
\begin{center}
\psfrag{a}{$p-1$}\psfrag{b}{$q-1$}\psfrag{c}{$p-2$}\psfrag{d}{$q-1$}\psfrag{e}{$q-2$}
\psfrag{A}{$D(p, q)$}\psfrag{B}{$B(p,q)$}\psfrag{C}{$B(2,q)$}
  \includegraphics[width=140mm]{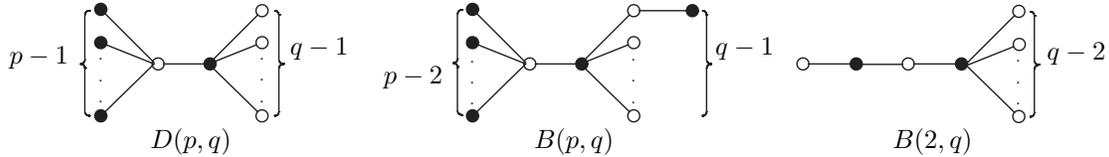}\\
  \caption{Trees $D(p, q), B(p, q)\,(p\ge q\ge 3)$ and $B(2, q).$ }
\end{center}
\end{figure}

\begin{thm}
For any $n\ge 2$.
\begin{kst}
\item[{\rm (i)}]Precisely the graph $D(p, q) \  (q\ge p\ge 1)$, which has
$2^{n-2}+2^{p-1}+2^{q-1}+n-2$ subtrees,
maximizes the total number of subtrees among $\mathscr{P}_n^{p,q}$.

\item[{\rm (ii)}]Precisely the graph $B(p, q) \  (q\ge p\ge 2)$, which has
$$\text{$3\cdot 2^{n-4}+3\cdot 2^{q-2}+2^{p-2}+n-1,$ if $p>2$ and $2^{n-2}+n+2$, otherwise}$$
subtrees, maximizes the total number of subtrees
among $\mathscr{P}_n^{p,q}\setminus \{D(p,q)\}.$
\item[{\rm (iii)}]Precisely the graph $P_{2p-1}(\lfloor\frac{n-2p+1}{2}\rfloor,\lceil\frac{n-2p+1}{2}\rceil)$, which has
$$(2p-1)\left(2^{\lfloor\frac{n-2p+1}{2}\rfloor}+2^{\lceil\frac{n-2p+1}{2}\rceil}\right)+{n-2p+1\choose{2}}+2^{n-2p+1}+n-2p+1$$ subtrees,
 minimizes the total number of subtrees among $\mathscr{P}_n^{p,q}$.
\end{kst}
\end{thm}

Given positive integers $n, q$ with $q\ge 2$, we call $T$ a \textit{complete $q$-ary tree} (or \textit{$q$-ary tree} for short) if any
non-pendant vertex $v$ in $T$ has exactly $q$ neighbours. Denote by $\mathscr{A}_n^q$ the class of $q$-ary trees with $n$ non-leaf
vertices ($(q-2)n+2$ leaves). Consider the path $P_{n+2}$ and attach $q-2$ pendant edges to each of the
non-leaf vertices of $P_{n+2}$. Denote the resulting tree by $\hat{T}_n^q$ (see Fig. 3).
It is easy to see that $\hat{T}_n^q \in \mathscr{A}_n^q$. In view of Theorem 2.3 in \cite{19}, it is easy to determine the tree in
$\mathscr{A}_n^q$ which maximizes the total number of subtrees. It is natural and interesting to determine the sharp lower bound on the total number of subtrees of trees among $\mathscr{A}_n^q.$
\begin{figure}[h!]
\begin{center}
\psfrag{i}{$v_i$}\psfrag{n}{$v_n$}\psfrag{a}{$v_{i-1}$}\psfrag{b}{$v_{i+1}$}
\psfrag{c}{$v_{n+1}$}\psfrag{d}{$q-2$}\psfrag{0}{$v_0$}\psfrag{1}{$v_1$}
\psfrag{1}{$v_1$}\psfrag{2}{$v_2$}\psfrag{A}{$\hat{T}_n^q$}\psfrag{B}{$T_{n,k}$}
\psfrag{3}{$v_{i-1}$}\psfrag{4}{$v_i$}\psfrag{5}{$v_{i+1}$}\psfrag{6}{$v_{n-k+1}$}\psfrag{k}{$k-1$}
  \includegraphics[width=80mm]{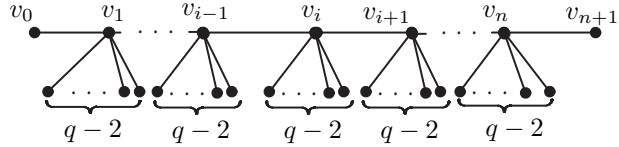}\\
  \caption{Tree $\hat{T}_n^q.$ }
\end{center}
\end{figure}
\begin{thm}
For $n\ge 1$, precisely the graph $\hat{T}_n^q$ (see Fig. 3), which has
$$
\frac{2^{q-2}(2^{q-1}-1)^2(2^{(n-1)(q-2)}-1)}{(2^{q-2}-1)^2}-\frac{n-1}{2^{q-2}-1}+2^q+nq-3n+3
$$
subtrees, minimizes the total number of subtrees among $\mathscr{A}_n^q.$
\end{thm}

\section{\normalsize Some Lemmas}
In this section, we give some necessary results which will be used to prove our main results.
For a set $S,$ let $|S|$ denote its cardinality. For two graphs $G_1, G_2$, if $G_1$ is a \textit{connected subgraph} of $G_2$, then we denote it by $G_1\subseteq G_2$.
Given a tree $T$ with $u,\,v\in V_T$, let
$$
\begin{array}{ll}
  \ \ \  f_T(u)  =  |\{T':\ \ T'\subseteq T, \, u\in V_{T'}\}|, & f_T(u*v) =  |\{T':\ \ T'\subseteq T, \, u,\,v\in V_{T'}\}|,\\
  f_T(u/v) =  |\{T':\ \ T'\subseteq T, \, u\in V_{T'}, v\not\in V_{T'}\}|, &\ \ \ \ \  F(T)=|\{T': \, T'\subseteq T,\, |V_{T'}|\ge 1\}|.
\end{array}
$$
\begin{lem}[\cite{21}]\label{lem2.1}
The $n$-vertex path $P_n$ has ${{n+1}\choose{2}}$ subtrees, fewer than any other tree on n
vertices.
\end{lem}
\begin{figure}[h!]
\begin{center}
\psfrag{a}{$x_1$}\psfrag{A}{$X_1$}
\psfrag{b}{$x_n$}\psfrag{B}{$X_n$}
\psfrag{c}{$y_n$}\psfrag{C}{$Y_n$}
\psfrag{d}{$y_1$}\psfrag{D}{$Y_1$}
\psfrag{y}{$y$}\psfrag{Y}{$Y$}\psfrag{z}{$z$}\psfrag{Z}{$Z$}\psfrag{x}{$x$}\psfrag{X}{$X$}
  \includegraphics[width=110mm]{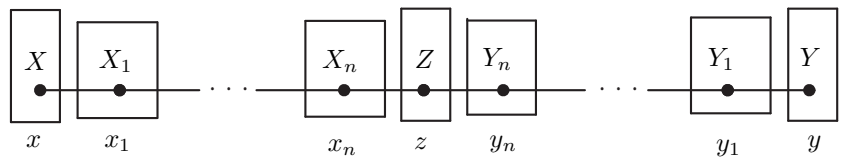}\\
  \caption{Path $P_W(x, y)$ connecting vertices $x$ and $y$. }
\end{center}
\end{figure}
Consider the tree $W$ in Fig. 4 with vertices $x$ and $y$, and
$$
P_W(x, y) = x_0(x)x_1 \ldots x_nzy_n \ldots y_1y_0(y)(x_0(x)x_1 \ldots x_ny_n \ldots y_1y_0(y))
$$
if $d_W(x, y)$ is even (odd) for any $n \geq 0$.
After the deletion of all the edges of $P_W(x, y)$ from $W$, some connected components
will remain. Let $X_i (X_0)$ denote the component that contains $x_i (x_0 = x)$, let
$Y_i (Y_0)$ denote the component that contains $y_i (y_0 = y)$, for $i = 1, 2, \ldots , n$, and let $Z$
denote the component that contains $z$.
\begin{lem}[\cite{20}]\label{lem2.2}
In the above situation, if $f_{X_i} (x_i) \geq f_{Y_i} (y_i)$ for $i = 0, 1, \ldots , n$, then
$f_W(x) \geq f_W(y)$. Furthermore, $f_W(x) = f_W(y)$ if and only if $f_{X_i} (x_i) = f_{Y_i} (y_i)$ for all
$i$.
\end{lem}
If we have a tree $T$ with vertices $x$ and $y$, and two rooted
trees $X$ and $Y$, then we can build two new trees, first $T'$, by
identifying the root of $X$ with $x$ and the root of $Y$ with $y$,
second $T''$, by identifying the root of $X$ with $y$ and the root of
$Y$ with $x$ (as shown in Fig. 5).
\begin{figure}[h!]
\begin{center}
\psfrag{a}{$T'$}\psfrag{b}{$T''$}\psfrag{T}{$T$}
\psfrag{y}{$y$}\psfrag{Y}{$Y$}\psfrag{x}{$x$}\psfrag{X}{$X$}
  \includegraphics[width=90mm]{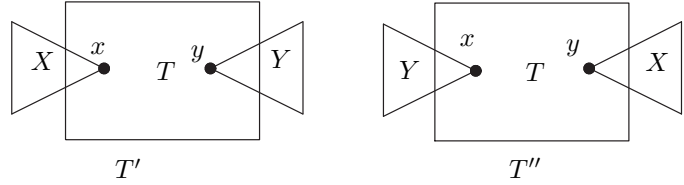}\\
  \caption{Switching subtrees rooted at $x$ and $y$.}
\end{center}
\end{figure}
\begin{lem}[\cite{20}]\label{lem2.3}
In the above situation, if $f_T(x)> f_T(y), f_X(x)< f_Y(y)$, then we have $F(T'')> F(T')$.
\end{lem}
\begin{cor}\label{cor2.4}
In the above situation, if $f_T(x)> f_T(y)$ and $X$ is a rooted tree that is not a single vertex, then we have $F(T'')> F(T')$,
 where $T'$ (resp. $T''$) is obtained by identifying the root of $X$ with $y$ (resp. $x$) of $T$.
\end{cor}
\begin{lem}\label{lem2.5}
Given an $n$-vertex path $P_n=v_1 v_2 \ldots v_n$, one has $f_{P_n}(v_k)=k(n-k+1)$ for $k\in \{1,2,\ldots, n\}$. Furthermore, one has
\begin{equation}\label{eq:2.1}
\begin{split}
f_{P_n}(v_k)&=f_{P_n}(v_{n-k+1}),\\
f_{P_n}(v_1)< f_{P_n}(v_2)<\cdots <f_{P_n}(v_k)&<f_{P_n}(v_{k+1}) <\cdots < f_{P_n}(v_{\lfloor
\frac{n+1}{2} \rfloor})=f_{P_n}(v_{\lceil \frac{n+1}{2} \rceil}).
\end{split}
\end{equation}
\end{lem}
\begin{proof}
For any $P \subseteq P_n$ such that $P$ contains $v_k$, it can be denoted by
$P=v_i v_{i+1} \ldots v_k\ldots v_j$, where $i\leq k \leq j$. It is easy to see that we have $i
\in \{1, 2, \ldots, k\}$ and $j \in \{k,k+1, \ldots, n\}.$ Hence, we
have $f_{P_n}(v_k)=k(n-k+1)$.

Consider the function $f(x)=x(n-x+1)$ for $x\ge 0$. By the monotonicity of $f(x)$, we have
$$
f(k)=f(n-k+1),\ \ \ \ f(1)<f(2)<\cdots<f(k)<f(k+1)<\cdots<f(\left\lfloor\frac{n+1}{2} \right\rfloor)=f(\left\lceil \frac{n+1}{2} \right\rceil),
$$
which is equivalent to
(\ref{eq:2.1}), as desired.
\end{proof}

By Corollary \ref{cor2.4} and Lemma \ref{lem2.5}, the following lemma follows immediately.
\begin{lem}\label{lem2.9}
Given a tree $T$ with at least two vertices and a path $P_k= v_1 v_2 \ldots
v_k$, let $T_i$ be a tree obtained from $T$ and $P_k$ by identifying a vertex $v$ of $T$ with
$v_i$ of $P_k$,\, $i\in \{2, 3,\ldots, \lfloor \frac{k+1}{2} \rfloor-2\}$. Then we have
$$F(T_i)=F(T_{k-i+1})$$
and $$
F(T_1)< F(T_2)<\cdots <F(T_i)  <\cdots < F(T_{\lfloor \frac{k+1}{2}\rfloor}).
$$
\end{lem}
\begin{lem}\label{lem2.6}
Given a tree $T$ with $uv\in E_T$ and $u$ is a leaf, one has $f_T(u)\leq
f_T(v)$, with equality if and only if $T\cong K_2$.
\end{lem}
\begin{proof}
For any edge $uv$ in $E_T$, we have
\[\label{eq:2-2}
  f_T(u)=f_T(u*v)+f_T(u/v),\ \ \ \ f_T(v)=f_T(u*v)+f_T(v/u).
\]
In particular, if $u$ is a leaf and $uv\in E_T$, then we have
$
  f_T(u/v)=1,\, f_T(v/u)= f_{T- u}(v)\geq 1,
$
with equality if and only if $T-u$ is a single vertex, i.e., $T\cong K_2$. Our result holds immediately.
\end{proof}
\begin{lem}\label{lem2.7}
Given a tree $T$ with $u,v\in V_T$ satisfying $f_T(u)\leq f_T(v),$ let $T'$ be a tree
obtained from $T$ by adding a new vertex $v_s$ to some vertex of $T$ such that in $T'$ the unique path between $u$
and $v_s$ contains $v,$ then $f_{T'}(u)< f_{T'}(v)$.
\end{lem}
\begin{proof}
Note that $f_T(u)\leq f_T(v)$, hence in view of (\ref{eq:2-2}), we have $f_T(u/v)\leq f_T(v/u)$. By the
structure of $T'$, it is straightforward to check that $f_{T'}(u/v)= f_{T}(u/v)$,
and $f_{T'}(v/u)> f_{T}(v/u)$, hence we have
$$f_{T'}(u)=
f_{T'}(u*v)+ f_{T'}(u/v)< f_{T'}(v*u)+ f_{T'}(v/u)= f_{T'}(v),
$$
as desired.
\end{proof}
\begin{lem}\label{lem2.8}
Let $P=uu_1\ldots v_1v$ be a path of a tree $T$ with $N_T(u)=\{u_1,w,w_1,\ldots, w_s\}$, $N_T(v)=\{v_1,z,z_1, \ldots, z_t\}$, here $s\geq 1, t\geq 1$. Then
$F(T)< F(T')$ or $F(T)< F(T'')$, where
\begin{eqnarray*}
  T'&=&T- uw_1- uw_2- \cdots-
uw_s+ vw_1+ vw_2+ \cdots+ uw_s, \\
  T''&=&T- vz_1- vz_2- \cdots-vz_t+
uz_1+ uz_2+ \cdots+ uz_t.
\end{eqnarray*}
\end{lem}
\begin{proof}
Consider the component in $T- uw_1- uw_2- \cdots- uw_s- vz_1- vz_2- \cdots-vz_t$, say $\hat{T}$, which contains both $u$ and $v$.
If $f_{\hat{T}}(u)\leq f_{\hat{T}}(v)$, then by Lemma \ref{lem2.7}, we have
$$
   f_{\tilde{T}}(u)< f_{\tilde{T}}(v),
$$ where $\tilde{T}$ is just the component containing both $u$ and $v$ in the graph $T- uw_1- uw_2- \cdots-
uw_s.$   By Corollary \ref{cor2.4},
we have $F(T)< F(T')$.

If $f_{\hat{T}}(v) < f_{\hat{T}}(u)$, similarly we can also show that $F(T)< F(T'')$. We omit the procedure here.

This completes the proof.
\end{proof}
\begin{lem}\label{lem2.10}
Given a tree $T$ containing a path $P_r=v_1v_2\ldots v_r,$  there exists a vertex $v_i \in V_{P_r}\setminus\{v_1, v_r\}$ such that
\begin{equation}\label{eq:2.3}
f_T(v_1)<\cdots < f_T(v_{i-1})< f_T(v_i)\geq f_T(v_{i+1})> \cdots >f_T(v_r).
\end{equation}
\end{lem}
\begin{proof}
Consider three vertices $x, y, z$ such that $xy, yz \in E_T$. Let $X,
Y , Z$, respectively, denote the components containing $x, y, z$ after the removal of the edges $xy$ and $yz$
from $T$. Observe the identities
\begin{eqnarray*}
  f_T(x)&=&f_X(x)+ f_X(x)f_Y(y)+ f_X(x)f_Y(y)f_Z(z),\\
  f_T(z)&=&f_Z(z)+ f_Z(z)f_Y(y)+ f_Z(z)f_Y(y)f_X(x),\\
  f_T(y)&=&f_Y (y)+ f_X(x)f_Y(y)+ f_Z(z)f_Y(y)+f_X(x)f_Y(y)f_Z(z).
\end{eqnarray*}
This gives
\begin{equation}\label{eq:2.4}
2f_T(y)-f_T(x)-f_T(z) = 2f_Y(y)+(f_X(x)+ f_Z(z))(f_Y(y)-1)> 0.
\end{equation}
Let
$$
 i=\min\{j:\ 1\le j\le r,\, f_T(v_j)\ge f_T(u), \, u\in V_{P_r}\}.
$$
by Lemma \ref{lem2.6}, $i\not=1, r.$ Hence, we have
$$
f_T(v_i)\geq f_T(v_{i+1}),\ \ \ f_T(v_i)> f_T(v_{i-1}).
$$

Next consider three consecutive vertices $v_i, v_{i+1},v_{i+2}$ on $P_r,$ in view of (\ref{eq:2.4}) we have
$$
2f_T(v_{i+1})-f_T(v_i)-f_T(v_{i+2})> 0.
$$
Combining with $f_T(v_i)\geq f_T(v_{i+1})$ yields
$$
f_T(v_i)\geq f_T(v_{i+1})>f_T(v_{i+2}).
$$
Repeated as above we obtain
\begin{equation}\label{eq:2.5}
f_T(v_i)\geq f_T(v_{i+1})> \cdots >f_T(v_r).
\end{equation}

Similarly, we obtain
\[\label{eq:2.6}
f_T(v_1)<\cdots < f_T(v_{i-1})< f_T(v_i).
\]
Hence, (\ref{eq:2.5}) and (\ref{eq:2.6}) imply (\ref{eq:2.3}) immediately.
\end{proof}
\section{\normalsize Proof of Theorem 1.1}\setcounter{equation}{0}

In this section, we shall determine sharp upper and lower bounds on the total number of subtrees of $n$-vertex tree with $k$ pendants.\vspace{2mm}

\noindent{\bf Proof of Theorem 1.1.}\ \ (i)\ First we enumerate the total number of subtrees of $T_n^k$.
Consider the unique vertex, say $v_0$, the center of $T_n^k$ whose degree is $k$, we have
\begin{align}\label{eq:3.1}
F(T_n^k)=f_{T_n^k}(v_0)+F(T-v_0)
        =f_{T_n^k}(v_0)+F(iP_{\lfloor\frac{n-1}{k}\rfloor}\cup jP_{\lceil\frac{n-1}{k}\rceil}),
\end{align}
where $i+j=k$ and $n-1\equiv j \pmod{k}$.
On the one hand,
$$
f_{T_n^k}(v_0)=\left(\left\lfloor\frac{n-1}{k}\right\rfloor+1\right)^i\left(\left\lceil\frac{n-1}{k}\right\rceil+1\right)^j.
$$
On the other hand, by Lemma \ref{lem2.1} we have
$$F(iP_{\lfloor\frac{n-1}{k}\rfloor}\cup jP_{\lceil\frac{n-1}{k}\rceil})=i{\lfloor\frac{n-1}{k}\rfloor+1\choose{2}}+j{\lceil\frac{n-1}{k}\rceil+1\choose{2}},$$
where $i+j=k$ and $j\equiv n-1\pmod{k}$. Together with (\ref{eq:3.1}), we have
$$F(T_n^k)=\left(\left\lfloor\frac{n-1}{k}\right\rfloor+1\right)^i\left(\left\lceil\frac{n-1}{k}\right\rceil+1\right)^j+i{\lfloor\frac{n-1}{k}\rfloor+1 \choose{2}}+j{\lceil\frac{n-1}{k}\rceil+1\choose{2}},$$
as desired.

Now we show that $T_n^k$ is the unique graph which maximizes the total number of subtrees among $\mathscr{T}_n^k.$ Choose $T\in \mathscr{T}_n^k$
such that the total number of its subtrees is as large as possible.

If $k=2$ or, $k=n-1$, it is easy to see that $\mathscr {T}_n^k=\{T_n^k\}$, our result follows immediately. Hence, in what follows we consider
$2<k<n-1$.

We are to show that $T$ is a spider, i.e., $T$ contains a unique vertex of degree larger than 2. Assume to the contrary that $T$
contains at least $2$ vertices of degree greater than 2. By Lemma 2.9, there exists an $n$-vertex tree $T'\in \mathscr{T}_n^k$
such that $F(T)<F(T')$, a contradiction to the choice of $T$. Hence, we assume that $v_0$ is the unique vertex of degree greater than 2.

In order to complete the proof, it suffices to show that all the legs attached to $v_0$ of $T$ are almost equal lengths.
Let $PV(T)=\{u_1, u_2, \ldots, u_k\}$ be the set of all pendant vertices of $T$. We are to show that for any $u_i, u_j\in PV(T)$, we have
$|d_T(v_0,u_i)-d_T(v_0,u_j)| \leq 1.$ Assume to the contrary that there exist two pendant vertices, say $u_t, u_l$, in $PV(T)$ such that
\begin{equation}\label{eq:3.2}
|d_T(u_0,u_t)-d_T(u_0,u_l)|
\ge 2.
\end{equation}
Denote the unique path connecting $u_t$ and $u_l$ by $P_s= w_1 w_2 \ldots w_{i-1}w_iw_{i+1}\ldots
w_s,$ where $w_1=u_t, w_s=u_l$ and $w_i=u_0, 1 \leq i \leq s$. In view of (\ref{eq:3.2}), we have
$$
\text{$u_0=w_i \neq w_{\lfloor\frac{s+1}{2}\rfloor}$\ \ \  and\ \ \ $u_0=w_i \neq w_{\lceil
\frac{s+1}{2}\rceil}$}.
$$
Hence, by Lemma 2.6 there exists an $n$-vertex tree $T'' \in \mathscr{T}_n^k$ such that
$F(T)<F(T'')$, a contradiction to the choice of $T$.

This completes the proof of Theorem 1.1(i).\vspace{2mm}

(ii)\ If $k=2$ or, $k=n-1$, it is easy to see that $\mathscr {T}_n^k=\{P_{n-k-1}(\lfloor\frac{k}{2}\rfloor,\lceil\frac{k}{2}\rceil)\}$, our result follows immediately. Hence, in what follows we consider $2<k<n-1$. In order to complete the proof, it suffices to show the following claims.

\begin{claim}
If $T$ minimizes the total number of subtrees in $\mathscr{T}_n^k, T\cong P_{n-k}(a,b)$, where $a \geq b \geq 1$ and $a+b=k$.
\end{claim}
\noindent{\bf Proof of Claim 1}\ \
If diam$(T)=3$, the claim follows immediately. Hence we consider the trees whose diameter is larger than 3.
Suppose that $P_r=v_1\ldots v_r (r \geq 5)$ is one of the longest path in $T$, we are to show that
$d_T(v_3)=d_T(v_4)=\cdots =d_T(v_{r-2})=2$. Assume to the contrary that there exists $v\in \{v_3, v_4, \ldots,v_{r-2}\}$ such that
$d_T(v)\ge 3.$ Let
$$
i=\min\{j:\ d_T(v_j)\ge 3,\ \ \ 3\leq j \leq r-2\},\ \ \ \ N_T(v_i)=\{v_{i-1}, v_{i+1}, z_1,z_2,\ldots, z_s\},\ \  s\geq1.
$$
After the deletion of all the vertices $z_1,z_2,\ldots, z_s$ from $T$, let $T_0$ denote the component containing $v_i$. By Lemma \ref{lem2.10}, there exists $v_t\in V_{P_r}$ such that
$$
f_{T_0}(v_1)<\cdots < f_{T_0}(v_{t-1})< f_{T_0}(v_t)\geq f_{T_0}(v_{t+1})> \cdots >f_{T_0}(v_r).
$$
If $t< i,$ then  we have $f_{T_0}(v_i)>f_{T_0}(v_{r-1})$.
By Corollary \ref{cor2.4}, we have
\[\label{eq:3.7}
 F(T) > F(T'),
\]
where
$$T'=T-v_iz_1-\cdots -v_iz_s+v_{r-1}z_1+\cdots+v_{r-1}z_s.$$

If $t \ge i,$ then  we have $f_{T_0}(v_i)>f_{T_0}(v_2)$.
By Corollary \ref{cor2.4}, we have
\[\label{eq:3.8}
 F(T) > F(T''),
\]
where
$$T''=T-v_iz_1-\cdots -v_iz_s+v_2z_1+\cdots+v_2z_s.$$

It is easy to see that $T',\, T''\in \mathscr{T}_n^k.$
Hence, (\ref{eq:3.7}) (resp. (\ref{eq:3.8})) is a contradiction to the choice of $T$. This completes the proof of Claim 1.
\qed
\begin{claim}
For positive integers $a,b$ with $a\geq b$ and $a+b=k$ one has
\begin{equation}\label{eq:3.9}
F(P_{n-k}(a,b))=(2^a+2^b)(n-k-1)+2^k+k+{n-k-1\choose{2}}
\end{equation}
and if $a-b\geq2$, then
\begin{equation}\label{eq:3.10}
F(P_{n-k}(a,b))> F(P_{n-k}(a-1,b+1)).
\end{equation}
\end{claim}
\noindent{\bf Proof of Claim 2}\ \
For convenience, assume that $d_{P_{n-k}(a,b)}(v_1)=a$ and $d_{P_{n-k}(a,b)}(v_{n-k})=b$. Then we have
\begin{equation}\label{eq:3.11}
\begin{split}
F(P_{n-k}(a,b))=&f_{P_{n-k}(a,b)}(v_1/v_{n-k})+f_{P_{n-k}(a,b)}(v_1*v_{n-k})+f_{P_{n-k}(a,b)}(v_{n-k}/v_1)\\
               &+F(P_{n-k}(a,b)-v_1-v_{n-k}).
\end{split}
\end{equation}

By direct calculation, we have
\[
   \text{$f_{P_{n-k}(a,b)}(v_1/v_{n-k})=2^a(n-k-1)$\ \ \  and\ \ \ $f_{P_{n-k}(a,b)}(v_{n-k}/v_1)=2^b(n-k-1)$.}
\]
It is straightforward to check that the total number of subtrees of $P_{n-k}(a,b)$ containing both $v_1$ and $v_{n-k}$
is equal to the total number of subtrees of $K_{1,a+b}$ each contains the center of $K_{1,a+b}$. Hence, we have
\[f_{P_{n-k}(a,b)}(v_1*v_{n-k})=2^{a+b}=2^k.
\]

On the other hand,
\[\label{eq:3.14}
F(P_{n-k}(a,b)-v_1-v_{n-k})=F((a+b)P_1\cup P_{n-k-2})=k+{n-k-1\choose{2}}.
\]

In view of (\ref{eq:3.11})-(\ref{eq:3.14}), (\ref{eq:3.9}) holds.

By (\ref{eq:3.9}), we have
$$
F(P_{n-k}(a,b))- F(P_{n-k}(a-1,b+1))=(2^a+2^b-2^{a-1}-2^{b+1})(n-k-1)=(2^{a-1}-2^b)(n-k-1).
$$
Note that $a-b\ge 2$, hence $(2^{a-1}-2^b)(n-k-1)>0$, i.e., $
F(P_{n-k}(a,b))>F(P_{n-k}(a-1,b+1))$, as desired.
\qed

By Claims 1 and 2, Theorem 1.1(ii) follows immediately.\qed

\section{\normalsize Proof of Theorem 1.2}\setcounter{equation}{0}
\noindent{\bf Proof of Theorem 1.2}\ In view of Theorem 3.7 in \cite{25}, we know that $\hat{T}_{n,i}^d$, $i=\lfloor\frac{d+2}{2}\rfloor=\lfloor\frac{d}{2}\rfloor+1$ or $i=\lceil\frac{d+2}{2}\rceil=\lceil\frac{d}{2}\rceil+1$, is the unique graph maximizing the total number of subtrees among $\mathscr{T}_{n,d}$. In order to complete the proof,
it suffices to show that
\[
F(\hat{T}_{n,i}^d)=2^{n-d-1}\left(\left\lfloor \frac{d}{2}\right \rfloor+1\right)\left(\left\lceil \frac{d}{2} \right\rceil+1\right)+ {\lfloor \frac{d}{2} \rfloor+1\choose{2}}{\lceil \frac{d}{2} \rceil+1\choose{2}}+n-d-1,
\]
where $i=\lfloor\frac{d}{2}\rfloor+1$ or $i=\lceil\frac{d}{2}\rceil+1$.

In fact, it is easy to see that
\[
   F(\hat{T}_{n,i}^d)=f_{\hat{T}_{n,i}^d}(x)+F(\hat{T}_{n,i}^d-x)
\]
for any $x \in V_{\hat{T}_{n,i}^d}$.

Without loss of generality, consider $x:=v_{\lfloor \frac{d}{2} \rfloor+1}$ whose neighbor contains just $n-d-1$ leaves. By Lemma \ref{lem2.5}, we have
$$
  f_{P_{d+1}}(x)=\left(\left\lfloor \frac{d}{2} \right\rfloor+1\right)\left(d+1+1-(\left\lfloor \frac{d}{2} \right\rfloor+1)\right)=\left(\left\lfloor \frac{d}{2}\right \rfloor+1\right)\left(\left\lceil \frac{d}{2} \right\rceil+1\right)
$$
which leads to
\[
  f_{\hat{T}_{n,i}^d}(x)=2^{n-d-1}\left(\left\lfloor \frac{d}{2}\right \rfloor+1\right)\left(\left\lceil \frac{d}{2} \right\rceil+1\right).
\]

On the other hand,
$$\hat{T}_{n,i}^d-x=(n-d-1)P_1\cup P_{\lfloor \frac{d
}{2} \rfloor}\cup P_{\lceil \frac{d}{2} \rceil}.$$
By Lemma \ref{lem2.1}, we have
\[F(\hat{T}_{n,i}^d-x)=n-d-1+{{\lfloor \frac{d}{2} \rfloor}+1\choose{2}}{{\lceil \frac{d}{2} \rceil}+1\choose{2}}.\]

In view of Eqs. (4.2)-(4.4), Eq. (4.1) follows immediately.
\qed

\section{\normalsize Proof of Theorem 1.3}\setcounter{equation}{0}
In this section, we prove Theorem 1.3. For convenience, denote by $\iota(T)$ the number of non-pendant vertices in $T$.\vspace{2mm}

{\noindent\bf Proof Theorem 1.3} (i)\ First we show that $D(p,q)$ is the tree in $\mathscr{P}_n^{p,q}$
which has the largest number of subtrees. For any $T\in \mathscr{P}_n^{p,q}$. If $p=1$, $\mathscr{P}_n^{p,q}= \{K_{1,n-1}\}= \{D(1,n-1)\}$.
Our result holds in this case. Hence, in what follows, we consider $p\ge 2.$ In order to determine the structure of the extremal graph, say $T$, in this case,
it suffice to show that $\iota(T)=2.$

Hence, we assume to the contrary that $\iota(T)\ge 3.$

Choose three vertices, say $u,v,w$, such that each of them is of degree at least 3. Let $V_T=V_1\cup V_2$. It is straightforward to check that in $\{u,v,w\}$, there exist two elements are in $V_1$ or $V_2$. We assume, without loss of generality, that $u,v \in V_1$ with $N_T(u)=\{u_1, z_1,\ldots ,z_t\}, N_T(v)=\{u_{2k-1}, r_1,\ldots ,r_s\}$, $t\ge 1,  s\ge 1$ and the unique path joining $u$ and $v$ is $P=uu_1\ldots u_{2k-1}v$. Let $X_u$ be the component that contains $u$ in $T-E_P$ and $Y_v$ the component that contains $v$ in $T-E_P$. Let $T'$ be the component that contains $u$ in $T- uz_1- \cdots -uz_t- vr_1- \cdots- vr_s$.

If $f_{T'}(u) \ge f_{T'}(v)$, by Lemma \ref{lem2.7} we have
\[
f_{T''}(u) > f_{T''}(v),
\]
where $T''$ is obtained by identifying $u$ of $T'$ with $u$ of $X_u$.
Let $T^*$ be the tree obtained by identifying $u$ of $T''$ with $v$ of $Y_v$. Note that $u$ and $v$ are in $V_1$, hence we have $T^* \in \mathscr{P}_n^{p,q}$. On the other hand, notice that $Y_v$ is not a single vertex, together with (5.1) and Corollary \ref{cor2.4}, we have $F(T)< F(T^*)$, a contradiction to the choice of $T$.

Similarly, if $f_{T'}(u) < f_{T'}(v)$, we can also show there exists a tree  $\hat{T} \in \mathscr{P}_n^{p,q}$ such that $F(T)< F(\hat{T})$, a contradiction. We omit the procedure here.
Hence, we get that $\iota(T)=2$, i.e., $T\cong D(p,q),$ as desired.

Now we show that
\begin{align}
 F(D(p,q)) = 2^{n-2}+2^{p-1}+2^{q-1}+n-2.
\end{align}

In fact, let $d_{D(p,q)}(u)=p$ and $d_{D(p,q)}(v)=q$, we have
$$
F(D(p,q)=f_{D(p,q)}(u/v)+f_{D(p,q)}(u*v)+f_{D(p,q)}(v/u)+F(D(p,q)-v-u).
$$
It is easy to see that $f_{D(p,q)}(u/v)=2^{p-1}, f_{D(p,q)}(v/u)=2^{q-1}$ and any subtree that contains both $u$ and $v$ can be considered be the subtree that contains the center of $K_{1,n-2}$, so we have $f_{D(p,q)}(u*v)=2^{n-2}$. On the other hand, as
$$
F(D(p,q)-v-u)=F((p-1)P_1\cup(q-1)P_1)=n-2
$$

Therefore, Eq.(5.2) holds. This completes the proof of Theorem 1.3 (i).\vspace{2mm}

(ii)\
Suppose $T'$ is the tree in $\mathscr{P}_n^{p,q} \setminus \{D(p, q)\}$ with $q\ge p\ge 2$ which has the largest number of subtrees. First we show that $\iota(T')= 3$.
Note that $T' \not \cong D(p, q)$, we know that $\iota(T) \neq 1, 2$. If $\iota(T') > 3$, then by a similar discussion as in the proof of Theorem 1.3 (i),
there exists a tree $T'' \in \mathscr{P}_n^{p,q}$ such that $\iota(T'')= \iota(T')- 1 \ge 3$ and $F(T') < F(T'')$. Note that $\iota(T'') \ge 3$, hence $T'' \in \mathscr{P}_n^{p,q}\setminus \{D(p, q)\}$. Therefore, we find a tree $T''$ in $\mathscr{P}_n^{p,q}\setminus \{D(p, q)\}$ such that $F(T') < F(T'')$, a contradiction to the choice of $T'$. Hence, $\iota(T') =3.$

Let $T(x, y, z)$ be the graph obtained by identifying one leaf of $K_{1, x+ 1}$ (resp. $K_{1, z+ 1}$) with the center of $K_{1, y}$, which is depicted in Fig. 6. As $T'\in \mathscr{P}_n^{p,q} \setminus \{D(p, q)\}$ with $\iota(T')=3$, we have
$$
  T'\cong T(a, p-2, b),\ \ \ a \ge b \ge 1, a+b+1=q,
$$
or
$$
  T'\cong T(a', q-2, b'),\ \ \ \ a' \ge b' \ge 1, a'+ b'+ 1=p.
$$
\begin{figure}[h!]
\begin{center}
  \psfrag{x}{$x$}\psfrag{y}{$y$}\psfrag{z}{$z$}
  \psfrag{a}{$a$}\psfrag{b}{$b$}\psfrag{v}{$v$}
  \psfrag{w}{$w$}\psfrag{u}{$u$}
\psfrag{c}{$T(x, y, z)$}\psfrag{d}{$T(a, 0, b)(a \ge b \ge 1)$}
  \includegraphics[width=110mm]{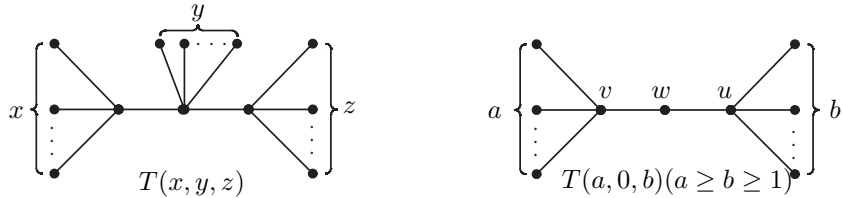}\\
  \caption{Trees $T(x, y, z)$ and $T(a, 0, b)(a \ge b \ge 1).$ }
\end{center}
\end{figure}
\setcounter{claim}{0}
\begin{claim}
If $T'\cong T(a, p-2, b)\ ($or $T'\cong T(a', q-2, b'))$, then $b=1$\ $($or $b'=1).$
\end{claim}
\noindent{\bf Proof of Claim 1}\ \
We only show that if $T'\cong T(a, p-2, b)$, we have $b=1.$ Similarly, we can also show if
$T'\cong T(a', q-2, b')$, then we have $b'=1$. We omit the procedure for the latter here.

Assume $b>1,$ let $u,v$ be two non-adjacent vertices of degree at least 2 in $T(a,p-2,b)$ with $N_{T(a,p-2,q)}(v)=\{w,w_1,\ldots,w_a\}$ and $N_{T(a,p-2,b)}(u)=\{w,z_1,z_2,\ldots, z_b\}$. By Lemma 2.2, as $f_{K_{1,a}}(v)\ge f_{K_{1,b}}(u)$, we have
\[
  f_{T(a,p-2,b)}(v)\ge f_{T(a,p-2,b)}(u).
\]
Let $T^*=T(a,p-2,b)-\{z_2,z_3,\ldots,z_b\}$. Then we have
\[
f_{T^*}(v)>f_{T^*}(u),
\]
otherwise $f_{T^*}(v)\le f_{T^*}(u)$. By Lemma \ref{lem2.7}, we have
$$
  f_{T(a,p-2,b)}(v)< f_{T(a,p-2,b)}(u),
$$
which contradicts (5.3). Hence, the inequality in (5.4) holds.

On the other hand, $T(a+b-1,p-2,1)$ can be obtained by identifying the vertex $v$ in $T^*$ with the center vertex of $K_{1,b-1}$; while
$T(a,p-2,b)$ can be obtained by identifying the vertex $u$ in $T^*$ with the center vertex of $K_{1,b-1}$. By Corollary \ref{cor2.4}, we have
$$
F(T(a+b-1,p-2,1))>F(T(a,p-2,b)),
$$
which contradicts the choice of $T' \,(=T(a,p-2,b)).$ Hence, $b=1$.
\qed

If $p= 2$, by Claim 1, we have $T'\cong D(2,q),$ or $T'\cong B(2,q)$. Note that $T'\in \mathscr{P}_n^{2,q}\setminus \{D(2,q\}$, hence $T'\cong B(2,q)$, as desired in this case.

If $p> 2$,  by Claim 1, we have $T'\cong T(q- 2, p-2, 1)$ or $T'\cong T(p-2, q-2, 1)=B(p,q)$. If $q=p,$ it is easy to see that
$T(q- 2, p-2, 1)\cong T(p-2, q-2, 1)=B(p,q),$ our result follows immediately in this subcase. Hence, it suffices to consider $q>p.$ In order to determine the
structure of $T'$,  it suffices to show that $F(T(p-2, q-2, 1))> F((q- 2, p- 2, 1)).$

Note that $T(q- 2, p- 2, 1)$ is obtained by identifying $u$ of $D(p- 1, q- 1)$ with a leaf of $P_3$, and $T(q- 2, p- 2, 1)$ is obtained by identifying $v$ of $D(q- 1, p- 1)$ with a leaf of $P_3$, where $u$ is a vertex of degree $p-1$ in $D(p- 1, q- 1)$, $v$ is a vertex of degree $q-1$ in $D(p- 1, q- 1)$.
Notice that
$$
  f_{D(p- 1, q- 1)}(u/v)= 2^{p- 2} < 2^{q- 2}=f_{D(p- 1, q- 1)}(v/u).
$$
So
$$
f_{D(p- 1, q- 1)}(u) < f_{D(p- 1, q- 1)}(v).
$$
Hence, by Corollary \ref{cor2.4} we have $F((q- 2, p- 2, 1))<F(T(p-2, q-2, 1))$.
Therefore, for any $T \in \mathscr{P}_n^{p,q} \setminus \{D(p, q)\}$, $F(T)\leq F(B(p, q)), q\ge p \ge 2$, with equality if and only if $T\cong B(p, q)$.

In order to complete the proof of Theorem 1.3 (ii), it suffices to show the following claim.
\begin{claim}
In the above situation, if $p\geq q\geq 2$ with $p+q=n$, then
\begin{equation*}
F(B(p,q))=\left\{ \begin{aligned}
&3 \cdot 2^{n-4}+3\cdot 2^{q-2}+2^{p-2}+n-1,     &q\geq p>2,   \\
&2^{n-2}+n+2,                         &p=2.
  \end{aligned} \right.
 \end{equation*}
\end{claim}
\noindent{\bf Proof of Claim 2}\ \
First consider $p=2$. Let $v$ be the vertex of degree $n-3$ in $B(2,n-2)$, hence
$F(B(2,n-2))=f_{B(2,n-2)}(v)+F(B(2,n-2)-v).$

Note that
$$
f_{B(2,n-2)}(v)=4\cdot 2^{n-2-2}=2^{n-2}
$$
and
$$
F(B(2,n-2)-v)=F(P_3\cup (n-4)P_1)={4\choose{2}}+n-4=n+2.
$$
Hence, by simple computation our result holds for $p=2$.

Now consider $p>2$. Let $v$ be the vertex of degree $q$ and $u$ the vertex of degree $p-1$ in $B(p,q)$. Note that

$$
F(B(p,q))=f_{B(p,q)}(v*u)+f_{B(p,q)}(v/u)+f_{B(p,q)}(u/v)+F(B(p,q)-v-u).
$$

On the other hand,
$$
f_{B(p,q)}(v*u)=3\cdot 2^{p-2+q-2}=3\cdot 2^{n-4},\ \ \ f_{B(p,q)}(v/u)=3\cdot 2^{q-2},\ \ \ f_{B(p,q)}(u/v)=2^{p-2}
$$
and
$$
F(B(p,q)-v-u)=F((p-2+q-2)P_1\cup P_2)=n-4+3=n-1.
$$

By simple calculation, our result also holds for $p>2.$ \qed

This completes the proof of Theorem 1.3 (ii).\vspace{2mm}

(iii)\ If $p=1$, it is easy to see that $\mathscr {T}_n^k=\{P_1(\lfloor\frac{n-1}{2})\rfloor,\lceil\frac{n-1}{2}\rceil\}$, our result follows immediately. On the other hand, if $p=q$ or $p=q-1$, it is easy to see that $P_n\in \mathscr{P}_n^{p,q}$, by Lemma \ref{lem2.1}, it is easy to see that $P_n=P_{2p-1}(1,1)$ or $P_n=P_{2p-1}(0,1),$ minimizes the total number of subtrees among $\mathscr{P}_n^{p,q}$. Hence, in what follows we consider $1<p< \lfloor\frac{n}{2}\rfloor$. In order to complete the proof, it suffices to show the following claim.
\begin{claim}
If $T$ minimizes the total number of subtrees in $\mathscr{P}_n^{p,q}, T\cong P_{2p-1}(a,b)$, where $a \geq b \geq 1$ and $a+b=n-2p+1$.
\end{claim}
\noindent{\bf Proof of Claim 3}\ \
If $1<p< \lfloor\frac{n}{2}\rfloor$, by (i) we know that $T \not \cong D(p,q)$, so diam$(T) \ge 3$. If diam$(T)=3$, the claim follows immediately. Hence in what follows we consider the trees whose diameter is larger than 3.
Suppose that $P_r=v_1\ldots v_r\, (r \geq 5)$ is one of the longest path in $T$, we are to show that
$d_T(v_3)=d_T(v_4)=\cdots =d_T(v_{r-2})=2$ and $r=2p+1$. First assume to the contrary that there exists $v\in \{v_3, v_4, \ldots,v_{r-2}\}$ such that
$d_T(v)\ge 3.$ Let
$$
i=\min\{j:\ d_T(v_j)\ge 3,\ \ \ 3\leq j \leq r-2\},\ \ \ \ N_T(v_i)=\{v_{i-1}, v_{i+1}, z_1,z_2,\ldots, z_s\},\ \  s\geq1.
$$
Let $T_0$ be the component that contains $v_i$ in $T-\{z_1,z_2,\ldots, z_s\}$. By Lemma \ref{lem2.10}, there exists $v_t\in V_{P_r}$ such that
$$
f_{T_0}(v_1)<\cdots < f_{T_0}(v_{t-1})< f_{T_0}(v_t)\geq f_{T_0}(v_{t+1})> \cdots >f_{T_0}(v_r).
$$
If $t< i,$ then we have $f_{T_0}(v_i)>f_{T_0}(v_{r-1})>f_{T_0}(v_r)$. If $v_i$ and $v_{r-1}$ are in the same part,
By Corollary \ref{cor2.4}, we have
\[\label{eq:5.14}
 F(T) > F(T'),
\]
where
$$T'=T-v_iz_1-\cdots -v_iz_s+v_{r-1}z_1+\cdots+v_{r-1}z_s,\ \ \ \ T'\in \mathscr{P}_n^{p,q}$$
otherwise, $v_i$ and $v_r$ are in the same part, we have
 \[\label{eq:5.15}
 F(T) > F(T''),
\]
where
$$T''=T-v_iz_1-\cdots -v_iz_s+v_rz_1+\cdots+v_rz_s,\ \ \ \ T''\in \mathscr{P}_n^{p,q}.$$
If $t \ge i$, repeat as above, we have a $T'''\in \mathscr{P}_n^{p,q}$ such that
\[\label{eq:5.16}
 F(T) > F(T'''),\ \ \ \ T'''\in \mathscr{P}_n^{p,q}.
\]
Hence, (\ref{eq:5.14})-(\ref{eq:5.16}) are contradictions to the choice of $T$. So we have $T\cong P_r(a,b)$.

On the other hand, since $T\in \mathscr{P}_n^{p,q}$ with $1<p<\lfloor\frac{n}{2}\rfloor$, it is easy to see that $r \le 2p+1$.
If $r< 2p+1$, it means that $v_1$ and $v_r$ are in different parts (otherwise we have $p< \lceil\frac{r-2}{2}\rceil$ or $q< \lceil\frac{r-2}{2}\rceil$). As $a \ge b$, we have $v_1 \in V_2$ and $v_r\in V_1$, where $V_1$ and $V_2$ are two parts of $V_T$ with $|V_1|=p, |V_2|=q$. Assume that $N_T(v_2)=\{v_3,v_1, w_2,\ldots, w_{a}\}, N_T(v_{r-1})=\{v_{r-2},v_r, u_2,\ldots, u_{b}\}$. Let $\hat{T}=T-\{v_r, u_2,\ldots, u_b\}$. By Lemma \ref{lem2.2}, we have $f_{\hat{T}}(v_{r-1}) < f_{\hat{T}}(v_1)$, by Corollary \ref{cor2.4} we have
\[\label{eq:5.18}
 F(T) > F(\tilde{T}),
\]
where
$$
  \tilde{T}=T-v_{r-1}v_r-v_{r-1}u_2-\cdots -v_ru_b+v_1v_r+v_1u_2+\cdots+v_1u_b.
$$
As $v_{r-1}, v_1\in V_2$, $\tilde{T}\in \mathscr{P}_n^{p,q}$. Hence, (\ref{eq:5.18}) is a contradiction to the choice of $T$. So we have $T\cong P_{2p-1}(a,b)$.
\qed

By Claims 2 in the proof of Theorem 1.1(ii), together with Claim 3, $T\cong P_{2p-1}(\lfloor\frac{n-2p+1}{2}\rfloor,\lceil\frac{n-2p+1}{2}\rceil)$, as desired.
\qed
\begin{remark} In view of Eq.(5.2), we have
$$
F(D(p,q))-F(D(p-1,q+1))=2^{p-2}-2^{q-1}<0
$$
for $q\geq p>1$. Hence, we have
\[
F(D(p,q))< F(D(p-1,q+1))< \cdots< F(D(1,n-1))=F(K_{1,n-1}),
\]
for $q\geq p>1$. Note that $D(p,q)$ maximizes the total number of subtrees among $\mathscr{P}_n^{p,q}$, hence in view of (5.9) and Theorem 1.3(i), the following corollary holds immediately.
\end{remark}
\begin{cor}[\cite{21}]
The star $K_{1,n-1}$ has $2^{n-1}+n-1$ subtrees, more than any other tree on $n$ vertices.
\end{cor}
\section{\normalsize Proof of Theorem 1.4}\setcounter{equation}{0}
In this section we shall determine the sharp lower bound on the total number of subtrees contained in a tree among $\mathscr{A}_n^q$.\vspace{2mm}

\noindent{\bf Proof of Theorem 1.4}\
First we characterize the structure of the tree, say $T$, minimizing the total number of subtrees in $\mathscr{A}_n^q.$
In order to do so, it suffices to show that the diameter of $T$ is $n+1$. Without loss of generality, we assume one of the longest paths in $T$ is $P_{r+1}=v_0 v_1\ldots v_r$. If $r=n+1$, our result holds obviously. So in what follows, we assume that $r\le n.$

For convenience, let
$$N_T(v_i)\setminus\{v_{i-1}, v_{i+1}\}=\{v_{i_1},v_{i_2},\ldots,v_{i_{q-2}}\},\ \ \ \ \ \  i= 1, 2, \ldots, r-1.$$
Hence,
\[
\bigcup_{i=1}^{r-1}N_T(v_i)=\{v_{1_1},v_{1_2},\ldots,v_{1_{q-2}},\ldots,v_{i_1},v_{i_2},\ldots,v_{i_{q-2}}, v_{r_1},v_{r_2},\ldots,v_{r_{q-2}}\}
\]
In fact, $(\bigcup_{i=1}^{r-1}N_T(v_i),\prec)$ is a total ordering set, where $\prec$ is defined as following: for any $v_{i_j},\, v_{t_s}\in \bigcup_{i=1}^{r-1}N_T(v_i)$, we call $v_{i_j}\prec v_{t_s}$ if $i<t$ or $i=t, j<s.$ Hence, we can order the elements in (6.1) as following:
$$
v_{1_1}\prec v_{1_2}\prec \cdots\prec v_{1_{q-2}}\prec \cdots\prec v_{i_1}\prec v_{i_2}\prec \cdots\prec v_{i_{q-2}}\prec v_{r_1}\prec v_{r_2}\prec \cdots\prec v_{r_{q-2}}.
$$
Note that $r<n+1$, hence there must exists non-pendant vertex in $\bigcup_{i=1}^{r-1}N_T(v_i).$ Choose the minimal element, say $v_{l_j}$, under the order $\prec$ such that it is a non-pendant vertex. Note that $P_{r+1}$ is the longest path in $T$, hence
\[
1 <l <r-1.
\]
Thus, we can partition $T$ into two subtrees, say $S$ and $T_0$, such that $E_T=E_S\cup E_{T_0}, V_T=V_S\cup V_{T_0}$ and $V_S\cap V_{T_0}=\{v_{l_j}\};$ see Fig. 7. For convenience, let $N_{T_0}(v_{l_j})=\{v_0,w_1,w_2,\ldots,w_{q-1}\}.$
\begin{figure}[h!]
\begin{center}
  \psfrag{0}{$v_0$}\psfrag{1}{$v_1$}\psfrag{2}{$v_2$}\psfrag{s}{$S$}
  \psfrag{l}{$v_l$}\psfrag{n}{$v_{r-1}$}\psfrag{c}{$v_r$}
  \psfrag{m}{$v_{l_j}$}\psfrag{e}{$T_0$}
\psfrag{d}{$q-2$}
  \includegraphics[width=70mm]{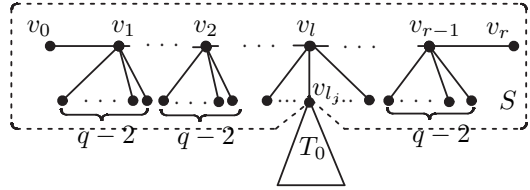}\\
  \caption{An $q$-arc tree $T$.}
\end{center}
\end{figure}

Now we are in the position to apply Lemma \ref{lem2.2} in the following setting:
$$
x\leftarrow v_0, \ x_i\leftarrow v_i, \ x_{\lfloor \frac{l}{2}\rfloor}\leftarrow v_{\lfloor \frac{l}{2}\rfloor}
$$
($z\leftarrow v_{\lceil \frac{l}{2}\rceil}$ if $l$ is odd)
$$
y\leftarrow v_{l_j}, \ y_i\leftarrow v_{l+1-i}, \ y_{\lfloor \frac{l}{2}\rfloor}\leftarrow v_{l+1-\lfloor \frac{l}{2}\rfloor}
$$
for $i = 1, 2, \ldots , \lfloor \frac{l}{2}\rfloor$.

Then
\begin{eqnarray*}
  X_i &=& K_{1,q-2},\ \ \ \ i = 1, 2, \ldots , \lfloor \frac{l}{2}\rfloor, \\
 Y_i &=& K_{1,q-2},\ \ \ \ i = 2, 3,\ldots , \lfloor \frac{l}{2}\rfloor
\end{eqnarray*}
and $Y_1$ is the component in $T-v_{l-1}v_l-v_lv_{l_j}$ which contains $v_l$. 
By direct calculation, it is easy to see that
$$
f_{X_i}(x_i)= f_{Y_i}(y_i)=2^{q-2}, \ \ \ \ i = 2, \ldots , \lfloor \frac{l}{2}\rfloor.
$$
On the other hand, in veiw of (6.2) we know that $Y_1\not\cong K_{1,q-2}$ and $f_{Y_1}(y_1)>2^{q-2},$ while
$f_{X_1}(x_1)=2^{q-2}.$ Therefore, we have
$$
f_{X_1}(x_1)<f_{Y_1}(y_1).
$$
By Lemma \ref{lem2.2}, we have
$$f_S(x) < f_S(y),$$
where $S$ is defined as above. Therefore, by Corollary \ref{cor2.4}, we have
\[
F(T')< F(T),
\]
where
$$
 T'=T-\{v_{l_j}w_1, v_{l_j}w_2,\ldots, v_{l_j}w_{q-1}\}+\{v_0w_1, v_0w_2,\ldots, v_0w_{q-1}\}.
$$
Inequality (6.3) is a contradiction to the choice of $T$. Hence, we obtain $r= n+1$, i.e., $T\cong \hat{T}_n^q,$ as desired.

In order to complete the proof of Theorem 1.4, it suffices to determine $F(\hat{T}_n^q)$. Choose one of the longest paths in
$\hat{T}_n^q$ and denote it by $P$. Let $v_0$ be one of its end-vertices. Then denote the unique neighbour of $v_0$ by $v_1$; see Fig. 3. Hence,
we have
\begin{eqnarray*}
      f_{\Hat{T}_n^q}(v_1)&=&2^{q-1}\left(1+2^{q-2}+\cdots+2^{(n-2)(q-2)}+ 2^{(n-1)(q-2)}+2^{(n-1)(q-2)}\right),\\
  f_{\Hat{T}_{n-1}^q}(v_1)&=&1\left(1+2^{q-2}+\cdots+2^{(n-2)(q-2)}+ 2^{(n-1)(q-2)}+2^{(n-1)(q-2)}\right).
\end{eqnarray*}
Note that
\begin{eqnarray*}
F(\Hat{T}_n^q)&=&f_{\Hat{T}_n^q}(v_1)+F(\Hat{T}_n^q-v_1) \\
&=&f_{\Hat{T}_n^q}(v_1)+F((q-1)P_1\cup (\Hat{T}_{n-1}^q-v_1))\\
&=&f_{\Hat{T}_n^q}(v_1)+F(\Hat{T}_{n-1}^q-v_1)+q-1\\
&=&f_{\Hat{T}_n^q}(v_1)+F(\Hat{T}_{n-1}^q)-f_{\Hat{T}_{n-1}^q}(v_1)+q-1.
\end{eqnarray*}
This gives
\begin{eqnarray*}
F(\Hat{T}_n^q)-F(\Hat{T}_{n-1}^q)&=&f_{\Hat{T}_n^q}(v_1)-f_{\Hat{T}_{n-1}^q}(v_1)+q-1\\
&=&q-1+(2^{q-1}-1)\left(1+2^{q-2}+\cdots+2^{(n-2)(q-2)}+ 2^{(n-1)(q-2)}+2^{(n-1)(q-2)}\right)\\
&=&q-1+(2^{q-1}-1)\left(\frac{2^{n(q-2)}-1}{2^{q-2}-1}+2^{(n-1)(q-2)}\right)\\
&=&q-1-\frac{2^{q-1}-1}{2^{q-2}-1}+\frac{2^{(n-1)(q-2)}(2^{q-1}-1)^2}{2^{q-2}-1}.
\end{eqnarray*}
As $F(\Hat{T}_1^q)=F(K_{1,q})=2^q+q$, we have
\begin{eqnarray*}
F(\Hat{T}_n^q)&=&F(B_1)+\sum_{i=2}^n\left[q-3-\frac{1}{2^{q-2}-1}+\frac{2^{(i-1)(q-2)}(2^{q-1}-1)^2}{2^{q-2}-1}\right]\\
&=&(2^q+q)+(n-1)(q-3)-\frac{n-1}{2^{q-2}-1}+\frac{(2^{q-1}-1)^2}{2^{q-2}-1}\sum_{i=2}^n2^{(i-1)(q-2)}\\
&=&\frac{2^{q-2}(2^{q-1}-1)^2(2^{(n-1)(q-2)}-1)}{(2^{q-2}-1)^2}-\frac{n-1}{2^{q-2}-1}+2^q+nq-3n+3.
\end{eqnarray*}
This completes the proof.
\qed

\begin{remark} In particular, let $q=3$ in Theorem 1.4, we can obtain that just the $n$-leaf binary caterpillar tree
minimizes the total number of subtrees among $n$-leaf binary trees, which is obtained by
Sz\'ekely and Wang in \cite{21}.
\end{remark}
\begin{cor}[\cite{21}]
For any $n \ge 2$, precisely the $n$-leaf binary caterpillar tree $\hat{T}_{n-2}^3,$ which has $2^{n+1}+
2^{n-2} - n- 4$ subtrees, minimizes the number of subtrees among $n$-leaf binary trees.
\end{cor}
\section{\normalsize Concluding remarks}
In view of Theorem 1.3, we conjecture that one may show the counterparts of these results for the Wiener index among the $n$-vertex trees with a given $(p,q)$-bipartition. On the other hand, for the Wiener index, sharp upper and lower bounds of trees with given degree sequence are determined; see \cite{18,26,27}. It is natural for us to determine sharp upper and lower bounds on the total number of subtrees of a tree with given degree sequence. It is difficult but interesting and it is still open.


\end{document}